\documentclass[12pt]{article}
\usepackage{amsmath,amssymb,amsthm}
\numberwithin{equation}{section}
\usepackage{units}
\usepackage{color}
\usepackage[T1]{fontenc}
\usepackage[utf8]{inputenc}
\usepackage{authblk}
\usepackage{bm}
\usepackage{graphicx}

\usepackage{datetime}

\usepackage[colorlinks=true,
linkcolor=webgreen,
filecolor=webbrown,
citecolor=webgreen]{hyperref}

\definecolor{webgreen}{rgb}{0,.5,0}
\definecolor{webbrown}{rgb}{.6,0,0}

\usepackage[toc,page]{appendix}

\newcommand{\red}[1]{{\color{red}#1}}

\newcommand{\Z}{{\mathbb Z}}

\newtheorem{thm}{Theorem}
\newtheorem{theorem}[thm]{Theorem}
\newtheorem{lemma}{Lemma}

\newtheorem{corollary}[thm]{Corollary}

\title{A master identity for Horadam numbers}

\author[]{Kunle Adegoke \\\href{mailto:adegoke00@gmail.com}{\tt kunle.adegoke@yandex.com}}

\affil{Department of Physics and Engineering Physics, \mbox{Obafemi Awolowo University}, 220005 Ile-Ife, Nigeria}

\begin{document}
\date{}

\maketitle

\begin{abstract}
\noindent We derive an identity involving Horadam numbers. Numerous new identities as well as those found in the existing literature are subsumed in this single identity.

\end{abstract}
\section{Introduction}
Our aim in writing this paper is to prove the presumably new identity
\[\tag{H}
u_{r - s} w_{n + m}  = u_{m - s} w_{n + r}  - q^{r - s} u_{m - r} w_{n + s}\,,
\]
where $r$, $s$, $n$ and $m$ are arbitrary integers and $(u_j(p,q))_{j\in\Z}$ and $(w_j(a,b;p,q))_{j\in\Z}$ are the Lucas sequence of the first kind  and the Horadam sequence, respectively, with complex parameters $p$ and $q$. Summation identities invoked by identity~(H) will also be derived. Numerous new identities as well as those derived by Horadam \cite{horadam65} and later researchers are subsumed in the single identity~(H).

\medskip

The Horadam sequence $(w_n) = \left(w_n(a,b; p, q)\right)$ is defined by the recurrence relation
\begin{equation}\label{eq.vhrb5b3}
w_0  = a,\,w_1  = b;\,w_n  = pw_{n - 1}  - qw_{n - 2}\, (n \ge 2)\,,
\end{equation}
where  $a$, $b$, $p$ and $q$ are arbitrary complex numbers, with $p\ne 0$ and $q\ne 0$. 

\medskip

Two important cases of  $(w_n)$ are the Lucas sequences of the first kind, $(u_n(p,q))=(w_n(0,1;p,q))$, and of the second kind, $(v_n(p,q))=(w_n(2,p;p,q))$; so that 
\begin{equation}\label{eq.s6z1bcx}
\mbox{$u_0=0$, $u_1=1$};\, u_n=pu_{n-1}-qu_{n-2}, \mbox{($n\ge 2$)};
\end{equation}
and
\begin{equation}\label{eq.zq47g0a}
\mbox{$v_0=2$, $v_1=p$};\, v_n=pv_{n-1}-qv_{n-2}, \mbox{($n\ge 2$)}.
\end{equation}
The most well-known Lucas sequences are the Fibonacci sequence, $(f_n)=(u_n(1,-1))$ and the sequence of Lucas numbers, $(l_n)=(v_n(1,-1)$.
The Fibonacci numbers, $f_n$, and the Lucas numbers, $l_n$ are defined by:
\begin{equation}
f_0  = 0,\;f_1  = 1,\;f_n  = f_{n - 1}  + f_{n - 2}\; (n\ge 2)
\end{equation}
and
\begin{equation}
l_0  = 2,\;l_1  = 1,\;l_n  = l_{n - 1}  + l_{n - 2}\; (n\ge 2)\,.
\end{equation}
Denote by $\alpha$ and $\beta$, the zeros of the characteristic polynomial $x^2-px+q$ for the Horadam sequence and the associated Lucas sequences. Then
\begin{equation}\label{eq.ufpxoag}
\alpha=\frac{p+\sqrt{p^2-4q}}2,\quad \beta=\frac{p-\sqrt{p^2-4q}}2\,,
\end{equation}
\begin{equation}\label{eq.dv3pphj}
\alpha+\beta=p,\quad\alpha-\beta=\sqrt{p^2-4q}\quad\mbox{and }\alpha\beta=q\,.
\end{equation}
The difference equations \eqref{eq.s6z1bcx}, \eqref{eq.zq47g0a} and \eqref{eq.vhrb5b3} are solved by the Binet-like formulas
\begin{equation}\label{eq.djr1ak1}
u_n=\frac{\alpha^n-\beta^n}{\alpha-\beta},\quad v_n=\alpha^n+\beta^n\,,
\end{equation}
and
\begin{equation}\label{eq.ffgsygd}
w_n  = b\left( {\frac{{\alpha ^n  - \beta ^n }}{{\alpha  - \beta }}} \right) - a\alpha \beta \left( {\frac{{\alpha ^{n - 1}  - \beta ^{n - 1} }}{{\alpha  - \beta }}} \right)\,.
\end{equation}
Thus,
\begin{equation}\label{eq.iae7rdq}
w_n=bu_n-aqu_{n-1}\,.
\end{equation}
It also follows that
\begin{equation}\label{eq.mdu5urs}
u_{2n}  = \frac{{\alpha ^{2n}  - \beta ^{2n} }}{{\alpha  - \beta }} = \left( {\frac{{\alpha ^n  - \beta ^n }}{{\alpha  - \beta }}} \right)\left( {\alpha ^n  + \beta ^n } \right) = u_n v_n\,.
\end{equation}
Setting $x=\alpha$ and $y=\beta$ in each of the following algebraic identities:
\begin{equation}\label{eq.jlhu2cf}
\frac{x^m - y^m}{x-y}(x^n  + y^n ) =\frac{x^{n + m}  - y^{n + m}}{x-y}  - x^m y^m\frac{x^{n - m}  - y^{n - m}}{x-y}\,,
\end{equation}
\begin{equation}
(x-y)^2\,\frac{x^m - y^m}{x-y}\,\frac{x^n  - y^n }{x-y} = x^{n + m}  + y^{n + m}  - x^m  y^m(x^{n - m}  + y^{n - m} )\,,
\end{equation}
\begin{equation}\label{eq.ibirhgo}
(x^m + y^m)\frac{x^n  - y^n }{x-y} = \frac{x^{n + m}  - y^{n + m}}{x-y}  + x^m  y^m\frac{x^{n - m}  - y^{n - m} }{x-y}
\end{equation}
and
\begin{equation}\label{eq.eu1f1ya}
(x^m + y^m)(x^n  + y^n ) = x^{n + m}  + y^{n + m}  + x^m  y^m(x^{n - m}  + y^{n - m} )\,,
\end{equation}
we find the following multiplication formulas for Lucas sequences:
\begin{equation}\label{eq.kb3hsvs}
u_mv_n = u_{n + m}  - q^m u_{n - m}\,,
\end{equation}
\begin{equation}\label{eq.nutjauf}
(p^2  - 4q)u_m u_n  = v_{n + m}  - q^m v_{n - m}\,,
\end{equation}
\begin{equation}\label{eq.ciqxfvx}
 v_m u_n = u_{n + m}  + q^m u_{n - m} 
\end{equation}
amd
\begin{equation}\label{eq.u8cxhf1}
v_m v_n = v_{n + m}  + q^m v_{n - m}\,.
\end{equation}
More properties of Lucas sequences can be found in the book by Ribenboim \cite[Chapter 1]{ribenboim}. The Mathworld \cite{mathworld_lucas} and Wikipedia \cite{wiki_lucas} articles are also good sources of information on the subject, with many references to useful materials. The books by Koshy \cite{koshy} and by Vajda \cite{vajda} are excellent reference materials on Fibonacci numbers and Lucas numbers.



\medskip

Extension of the definition of $w_n$ to negative subscripts is provided by writing the recurrence relation as $w_{-n}=(pw_{-n+1}-w_{-n+2})/q$. Using the Binet-like formulas \eqref{eq.djr1ak1} and \eqref{eq.ffgsygd}, it is readily established that
\begin{equation}\label{eq.t2c6ndd}
u_{-n}=-u_n/q^n\,,\quad v_{-n}=v_n/q^n\,.
\end{equation}
From \eqref{eq.iae7rdq} and \eqref{eq.t2c6ndd}, it follows that
\begin{equation}
w_{ - n}  = \frac{{w_{ - n} }}{{w_n }}w_n  = \left( {\frac{{(ap - b)u_n  - aqu_{n - 1} }}{{bu_n  - aqu_{n - 1} }}} \right)q^{ - n} w_n\,.
\end{equation}
We require the following very general summation identities.
\begin{lemma}[{\cite[Lemma 1]{adegoke18}}]\label{lem.u4bqbkc}
Let $(X_n)$ and $(Y_n)$ be any two sequences such that $X_n$ and $Y_n$, $n\in\Z$, are connected by a three-term recurrence relation $hX_n=f_1X_{n-c}+f_2Y_{n-d}$, where $h$, $f_1$ and $f_2$ are arbitrary non-vanishing complex functions, not dependent on $n$, and $c$ and $d$ are integers. Then, the following identity holds for integer $k$:
\[
f_2 \sum_{j = 0}^k {f_1^{k - j} h^j Y_{n - kc  - d  + c j} }  = h^{k + 1} X_n  - f_1^{k + 1} X_{n - (k + 1)c }\,.
\]

\end{lemma}
\begin{lemma}[{\cite[Lemma 2]{adegoke18}}]\label{lem.s9jfs7n}
Let $(X_n)$ be any arbitrary sequence, where $X_n$, $n\in\Z$, satisfies a three-term recurrence relation $hX_n=f_1X_{n-c}+f_2X_{n-d}$, where $h$, $f_1$ and $f_2$ are arbitrary non-vanishing complex functions, not dependent on $n$, and $c$ and $d$ are integers. Then, the following identities hold for integer $k$:
\begin{equation}\label{eq.mxyb9zk}
f_2 \sum_{j = 0}^k {f_1^{k - j} h^j X_{n - kc  - d  + c j} }  = h^{k + 1} X_n  - f_1^{k + 1} X_{n - (k + 1)c }\,,
\end{equation}
\begin{equation}\label{eq.cgldajj}
f_1 \sum_{j = 0}^k {f_2^{k - j} h^j X_{n - kd  - c  + d j} }  = h^{k + 1} X_n  - f_2^{k + 1} X_{n - (k + 1)d }
\end{equation}
and
\begin{equation}\label{eq.n2n4ec3}
h\sum_{j = 0}^k {( - 1)^j f_2^{k - j} f_1 ^j X_{n - (d  - c )k + c  + (d  - c )j} }  = ( - 1)^k f_1 ^{k + 1} X_n  + f_2^{k + 1} X_{n - (d  - c )(k + 1)}\,.
\end{equation}
\end{lemma}
\begin{lemma}[{\cite[Lemma 3]{adegoke18}}]\label{lem.binomial}
Let $(X_n)$ be any arbitrary sequence. Let $X_n$, $n\in\Z$, satisfy a three-term recurrence relation $hX_n=f_1X_{n-c}+f_2X_{n-d}$, where $h$, $f_1$ and $f_2$ are non-vanishing complex functions, not dependent on $n$, and $c$ and $d$ are integers. Then,
\begin{equation}\label{eq.fe496kc}
\sum_{j = 0}^k {\binom kjf_2^{k - j} f_1^j X_{n - d k + (d  - c )j} }  = h^k X_n\,,
\end{equation}
\begin{equation}\label{eq.j7k6a8g}
\sum_{j = 0}^k {( - 1)^j \binom kjf_2^{k - j} h^j X_{n + (c  - d )k + d j} }  = ( - 1)^k f_1^k X_n
\end{equation}
and
\begin{equation}\label{eq.fnwrzi3}
\sum_{j = 0}^k {( - 1)^j \binom kjf_1^{k - j} h^j X_{n + (d  - c )k + c j} }  = ( - 1)^k f_2^k X_n\,,
\end{equation}
for $k$ a non-negative integer.

\end{lemma}
\begin{lemma}\label{lem.reciprocal1}
Let $(X_n)$ and $(Y_n)$ be any two sequences such that $X_n$ and $Y_n$, $n\in\Z$, are connected by a three-term recurrence relation $hX_n=f_1X_{n-c}+f_2Y_{n-d}$, where $f_1$ and $f_2$ are arbitrary non-vanishing complex functions, not dependent on $n$, and $c$, $d$ and $k$ are integers. Then,
\begin{equation}
X_nX_{n - c (k + 1)} f_2 \sum_{j = 0}^k h^{k-j}f_1^j{\frac{{Y_{n - d - c k + c j} }}{{X_{n - c k + c j} X_{n - c - c k + c j} }}}  = h^{k+1}X_n - f_1^{k + 1}X_{n - c (k + 1)}\,.
\end{equation}

\end{lemma}
\begin{lemma}\label{lem.reciprocal}
Let $(X_n)$ be any arbitrary sequence. Let $X_n$, $n\in\Z$, satisfy a three-term recurrence relation $hX_n=f_1X_{n-c}+f_2X_{n-d}$, where $f_1$ and $f_2$ are non-vanishing complex functions, not dependent on $n$, and $c$, $d$ and $k$ are integers. Then, the following identities hold for arbitrary integers $n$, $c$, $d$ and $k$ for which the summand is not singular in the summation interval:
\begin{equation}\label{eq.q1q4vob}
X_nX_{n - c (k + 1)} f_2 \sum_{j = 0}^k h^{k-j}f_1^j{\frac{{X_{n - d - c k + c j} }}{{X_{n - c k + c j} X_{n - c - c k + c j} }}}  = h^{k+1}X_n - f_1^{k + 1}X_{n - c (k + 1)}\,,
\end{equation}
\begin{equation}\label{eq.mg9cdve}
X_nX_{n - d (k + 1)} f_1 \sum_{j = 0}^k h^{k-j}f_2^j{\frac{{X_{n - c - d k + d j} }}{{X_{n - d k + d j} X_{n - d - d k + d j} }}}  = h^{k+1}X_n - f_2^{k + 1}X_{n - d (k + 1)}\,,
\end{equation}
and
\begin{equation}\label{eq.l8b84rc}
\begin{split}
X_nX_{n - (d - c)(k + 1)}h&\sum_{j = 0}^k {(-1)^jf_1^{k-j}f_2^j\frac{{X_{n + c - (d - c)k + (d - c)j} }}{{X_{n - (d - c)k + (d - c)j} X_{n - d + c - (d - c)k + (d - c)j} }}}\\
&\qquad= f_1^{k+1}X_n + (-1)^kf_2^{k+1}X_{n - (d - c)(k + 1)}\,.
\end{split}
\end{equation}

\end{lemma}
\section{The master identity and some consequences}
\begin{theorem}
The following identity holds for arbitrary integers $r$, $s$, $n$ and $m$:
\[
u_{r - s} w_{n + m}  = u_{m - s} w_{n + r}  - q^{r - s} u_{m - r} w_{n + s}\,.
\]
\end{theorem}
\begin{proof}
Since both sequences $(u_n)$ and $(w_n)$ have the same recurrence relation, we choose a basis set in $(u_n)$ and express the numbers from $(w_n)$ in this basis. Let
\begin{equation}\label{eq.nbl33qg}
w_{n+m}=\lambda_1u_{m-s}+\lambda_2u_{m-r}\,,
\end{equation}
where $r$, $s$, $n$ and $m$ are arbitrary integers and the coefficients $\lambda_1$ and $\lambda_2$ are to be determined. Setting $m=r$ and $m=s$, in turn, gives
\begin{equation}\label{eq.jcs53o2}
w_{n + r}  = \lambda_1 u_{r - s},\quad w_{n + s}  = \lambda_2 u_{s - r}\,.
\end{equation}
Multiplying through identity \eqref{eq.nbl33qg} by $u_{r-s}u_{s-r}$ gives
\begin{equation}
u_{r - s} u_{s - r} w_{n + m}  = \underbrace{\lambda _1 u_{r - s}}_{w_{n + r}} u_{s - r} u_{m - s}  + \underbrace{\lambda _2 u_{s - r}}_{w_{n + s}} u_{r - s} u_{m - r}\,.
\end{equation}
Thus, we find
\[
\begin{split}
u_{r - s} u_{s - r} w_{n + m}  &= u_{s - r} u_{m - s} w_{n + r}  + u_{r - s} u_{m - r} w_{n + s}\\
&= u_{s - r} u_{m - s} w_{n + r}  - q^{r - s} u_{s - r} u_{m - r} w_{n + s}\,;
\end{split}
\]
so that identity~(H) is satisfied identically if $r=s$ and numerically if $r\ne s$.
\end{proof}
Since $r$, $s$, $m$ and $n$ are arbitrary, identity~(H) also implies the following identities:
\[\tag{F}
u_{r - s} w_{n + m}  = u_{n - s} w_{m + r}  - q^{r - s} u_{n - r} w_{m + s}\,,
\]
\[\tag{G}
u_{r - s} w_{n + m}  = u_{n + r} w_{m - s}  - q^{r - s} u_{n + s} w_{m - r}
\]
and
\[\tag{J}
u_{{r-s}}w_{{n+m}}=u_{{m+r}}w_{{n-s}}-{q}^{r-s}u_{{m+s}}w_{{n-r}}\,.
\]
In the particular cases of Lucas sequences, identities (H), (F), (G) and (J) read:
\begin{equation}
u_{r - s} u_{n + m}  = u_{m - s} u_{n + r}  - q^{r - s} u_{m - r} u_{n + s}\,,
\end{equation}
\begin{equation}
u_{r - s} u_{n + m}  = u_{n - s} u_{m + r}  - q^{r - s} u_{n - r} u_{m + s}\,,
\end{equation}
\begin{equation}
u_{r - s} u_{n + m}  = u_{n + r} u_{m - s}  - q^{r - s} u_{n + s} u_{m - r}\,,
\end{equation}
\begin{equation}
u_{{r-s}}u_{{n+m}}=u_{{m+r}}u_{{n-s}}-{q}^{r-s}u_{{m+s}}u_{{n-r}}\,;
\end{equation}
and
\begin{equation}
u_{r - s} v_{n + m}  = u_{m - s} v_{n + r}  - q^{r - s} u_{m - r} v_{n + s}\,,
\end{equation}
\begin{equation}
u_{r - s} v_{n + m}  = u_{n - s} v_{m + r}  - q^{r - s} u_{n - r} v_{m + s}\,,
\end{equation}
\begin{equation}
u_{r - s} v_{n + m}  = u_{n + r} v_{m - s}  - q^{r - s} u_{n + s} v_{m - r}\,,
\end{equation}
\begin{equation}
u_{{r-s}}v_{{n+m}}=u_{{m+r}}v_{{n-s}}-{q}^{r-s}u_{{m+s}}v_{{n-r}}\,.
\end{equation}
Identities (H), (F), (G) and (J) have unlimited consequences. We list a few.
\begin{corollary}\label{cor.y980xry}
The following identities hold for integers $n$, $m$, $j$, $r$, $s$ and $t$:
\begin{equation}\label{eq.h1s1us8}
\text{$r=0$, $s=-m$ in (H) }\red{\implies} v_m w_n  = w_{n + m}  + q^m w_{n - m}\,,
\end{equation}
\begin{equation}\label{eq.zur6cp7}
v_n w_n  = w_{2n}  + q^n a\,,
\end{equation}
\begin{equation}\label{eq.e7izv7b}
\text{$r=0$, $s=-m$ in (F) }\red{\implies} u_m w_n  = u_n w_m  - q^m au_{n - m}\,,
\end{equation}
\begin{equation}\label{eq.fgo79vj}
\text{$r=1$, $s=0$ in (H) }\red{\implies} w_{n + m}  = u_m w_{n + 1}  - qu_{m - 1} w_n\,, 
\end{equation}
\begin{equation}\label{eq.hfbzyp7}
\text{$m\to -m$ in \eqref{eq.fgo79vj}} \red{\implies} q^m w_{n - m}  = u_{m +1} w_{n}  - u_{m} w_{n+1}\,, 
\end{equation}
\begin{equation}
\text{\eqref{eq.fgo79vj} $-$ \eqref{eq.hfbzyp7}} \red{\implies} w_{n + m}  - q^m w_{n - m}  = u_m (w_{n + 1}  - qw_{n - 1} )\,,
\end{equation}
\begin{equation}\label{eq.h2p9o4s}
w_{n + m}  = u_n w_{m + 1}  - qu_{n - 1} w_m\,, 
\end{equation}
\begin{equation}
w_{n + m}  = u_{m - j} w_{n + j + 1}  - qu_{m - j - 1} w_{n + j}\,,
\end{equation}
\begin{equation}
w_{n + m}  = u_{n - j} w_{m + j + 1}  - qu_{n - j - 1} w_{m + j}\,,
\end{equation}
\begin{equation}
w_{2n}  = u_n w_{n + 1}  - qu_{n - 1} w_n 
\end{equation}
\begin{equation}
w_{2n}  = u_{n + 1} w_n  - qu_n w_{n - 1}\,,
\end{equation}
\begin{equation}
w_{2n - 1}  = u_{n + 1} w_{n - 1}  - qu_n w_{n - 2}\,,
\end{equation}
\begin{equation}
w_{2n - 1}  = u_n w_n  - qu_{n - 1} w_{n - 1}\,,
\end{equation}
\begin{equation}\label{eq.c8nz8uo}
\text{$r=0$, $s=m-n$ in (H) }\red{\implies} u_{n - m} w_{n + m}  = u_n w_n  - q^{n - m} u_m w_m\,,
\end{equation}
\begin{equation}\label{eq.da86tc0}
\text{$r=0$, $s=-m$ in (F) }\red{\implies} u_{n - m} w_{n + m}  = u_{2n - m} w_m  - q^{n - m} u_n w_{2m - n}\,,
\end{equation}
\begin{equation}\label{eq.xzxwnfu}
\text{$m=0$ in \eqref{eq.da86tc0} }\red{\implies} q^n w_{ - n}  = av_n  - w_n\,,
\end{equation}
\begin{equation}\label{eq.wcmvj20}
\text{use of \eqref{eq.xzxwnfu} in \eqref{eq.h1s1us8}} \red{\implies} v_n w_m  - aq^mv_{n - m}  = w_{n + m}  - q^m w_{n - m}\,,
\end{equation}
\begin{equation}
\text{\eqref{eq.h1s1us8}$\times$\eqref{eq.wcmvj20}} \red{\implies} w_{n + m}^2  - q^{2m} w_{n - m}^2  = v_m w_n (v_n w_m  - aq^mv_{n - m} )\,,
\end{equation}
\begin{equation}\label{eq.wiwsjz4}
\text{$s=-r$ in (H) }\red{\implies} u_{2r} w_{n + m}  = u_{m + r} w_{n + r}  - q^{2r} u_{m - r} w_{n - r}\,,
\end{equation}
\begin{equation}
\text{$m\to -m$ in \eqref{eq.wiwsjz4}} \red{\implies} q^{m - r} u_{2r} w_{n - m}  = u_{m + r} w_{n - r}  - u_{m - r} w_{n + r}\,,
\end{equation}
\begin{equation}\label{eq.efn31zb}
u_{2r} w_{n + m}  = u_{n + r} w_{m + r}  - q^{2r} u_{n - r} w_{m - r}\,,
\end{equation}
\begin{equation}
u_{2r} w_{2n}  = u_{n + r} w_{n + r}  - q^{2r} u_{n - r} w_{n - r}\,,
\end{equation}
\begin{equation}
u_{2r} w_{2n - 1}  = u_{n + r} w_{n + r - 1}  - q^{2r} u_{n - r} w_{n - r - 1}\,,
\end{equation}
\begin{equation}
\text{$r=1$ in \eqref{eq.wiwsjz4} }\red{\implies} pw_{n + m}  = u_{m + 1} w_{n + 1}  - q^2 u_{m - 1} w_{n - 1}\,,
\end{equation}
\begin{equation}
pw_{n + m}  = u_{n + 1} w_{m + 1}  - q^2 u_{n - 1} w_{m - 1}\,,
\end{equation}
\begin{equation}
pw_{2n}  = u_{n + 1} w_{n + 1}  - q^2 u_{n - 1} w_{n - 1}\,,
\end{equation}
\begin{equation}
pw_{2n - 1}  = u_{n + 1} w_n  - q^2 u_{n - 1} w_{n - 2}\,,
\end{equation}
\begin{equation}\label{eq.ce4ih96}
\text{$m=0$, $r=t+s$ in (H) }\red{\implies} u_{{t}}w_{{n}}=u_{{s}}w_{{n+t-s}}-{q}^{t}u_{{s-t}}w_{{n-s}}\,,
\end{equation}
\begin{equation}
\text{$m=0$, $r=t+s$ in (F) }\red{\implies} u_{{t}}w_{{n}}=u_{{n-s}}w_{{t+s}}-{q}^{t}u_{{n-t-s}}w_{{s}}\,,
\end{equation}
\begin{equation}\label{eq.cz2qevd}
\text{$m=0$, $r=t+s$ in (G) }\red{\implies} u_{{t}}w_{{n}}=u_{{n+t-s}}w_{{s}}-{q}^{t}u_{{n-s}}w_{{s-t}}\,,
\end{equation}
\begin{equation}
\text{$m=0$, $r=t+s$ in (J) }\red{\implies} u_{{t}}w_{{n}}=u_{{t+s}}w_{{n-s}}-{q}^{t}u_{{s}}w_{{n-s-t}}\,.
\end{equation}
\end{corollary}
The identities in Corollary \ref{cor.y980xry} have interesting implications for the Lucas sequences. The following list is far from being exhaustive.
\begin{corollary}
The following identities hold for integers $n$, $m$, $r$, $s$ and $t$:
\begin{equation}
\text{$w_n=v_n$ in \eqref{eq.zur6cp7}}\red{\implies} v_n^2  = v_{2n}  + 2q^n \,,
\end{equation}
\begin{equation}\label{eq.bymjqnv}
\text{$w_n=v_n$ in \eqref{eq.e7izv7b}}\red{\implies} u_n v_m  - u_m v_n  = 2q^m u_{n - m} \,,
\end{equation}
\begin{equation}
\text{$w_n=v_n$, $m=0$ in \eqref{eq.h2p9o4s}}\red{\implies} v_n  = pu_n  - 2qu_{n - 1} \,,
\end{equation}
\begin{equation}
u_{n + m}  = u_m u_{n + 1}  - qu_{m - 1} u_n \,,
\end{equation}
\begin{equation}
v_{n + m}  = u_m v_{n + 1}  - qu_{m - 1} v_n \,,
\end{equation}
\begin{equation}
u_{2m - 1}  = u_m^2  - qu_{m - 1}^2 \,,
\end{equation}
\begin{equation}
v_{2m - 1}  = u_{2m}  - qu_{2m - 2} \,,
\end{equation}
\begin{equation}
\text{$w_n=u_n$ in \eqref{eq.c8nz8uo}}\red{\implies} u_{n - m} u_{n + m}  = u_n^2  - q^{n - m} u_m^2 \,,
\end{equation}
\begin{equation}
\text{$w_n=v_n$ in \eqref{eq.c8nz8uo}}\red{\implies} u_{n - m} v_{n - m}  = u_{2n}  - q^{n - m} u_{2m} \,,
\end{equation}
\begin{equation}
\text{\eqref{eq.nutjauf} and $w_n=v_n$ in \eqref{eq.wcmvj20}} \red{\implies} v_nv_m-(p^2-4q)u_mu_n=2q^mv_{n-m}\,,
\end{equation}
\begin{equation}
\text{$w_n=u_n$ in \eqref{eq.efn31zb}}\red{\implies} u_{2r} u_{n + m}  = u_{n + r} u_{m + r}  - q^{2r} u_{m - r} u_{n - r} \,,
\end{equation}
\begin{equation}
\text{$w_n=v_n$ in \eqref{eq.efn31zb}}\red{\implies} u_{2r} v_{n + m}  = u_{n + r} v_{m + r}  - q^{2r} u_{m - r} v_{n - r} \,,
\end{equation}
\begin{equation}
u_{2r} u_{2n}  = u_{n + r}^2  - q^{2r} u_{n - r}^2 \,,
\end{equation}
\begin{equation}
u_{2r} v_{2n}  = u_{2(n + r)}  - q^{2r} u_{2(n - r)} \,,
\end{equation}
\begin{equation}
pu_{2n}  = u_{n + 1}^2  - q^2 u_{n - 1}^2 \,,
\end{equation}
\begin{equation}
pv_{2n}  = u_{2(n + 1)}  - q^2 u_{2(n - 1)} \,,
\end{equation}
\begin{equation}
\text{$w_n=u_n$ in \eqref{eq.ce4ih96}}\red{\implies} u_t u_n  = u_s u_{n + t - s}  - q^t u_{s - t} u_{n - s} \,,
\end{equation}
\begin{equation}
\text{$w_n=v_n$ in \eqref{eq.ce4ih96}}\red{\implies} u_t v_n  = u_s v_{n + t - s}  - q^t u_{s - t} v_{n - s} \,,
\end{equation}
\begin{equation}\label{eq.fhr40g7}
\text{$s=0$, $w_n=v_n$ in \eqref{eq.cz2qevd}}\red{\implies} u_n v_t  + u_t v_n  = 2u_{n + t}\,,
\end{equation}
\begin{equation}
\text{$m=t$ in \eqref{eq.bymjqnv} $\times$ \eqref{eq.fhr40g7}}\red{\implies} u_n^2 v_t^2  - u_t^2 v_n^2  = 4q^t u_{n + t} u_{n - t} \,,
\end{equation}
\begin{equation}
p^2u_n^2 - v_n^2  = 4q u_{n + 1} u_{n - 1}\,.
\end{equation}
\end{corollary}
\section{Summation identities involving binomial coefficients}
\begin{theorem}\label{thm.tutbc4x}
The following identities hold for positive integer $k$ and arbitrary integers $r$, $s$, $n$, $m$:
\begin{equation}\label{eq.a4qiltd}
\sum_{j = 0}^k {  (-1)^jq^{(r-s)(k-j)} \binom kju_{m - s}^j u_{m - r}^{k - j} w_{n - (m - s)k + (r - s)j} }  =(-1)^k u_{r - s}^k w_n\,,
\end{equation}
\begin{equation}
\sum_{j = 0}^k {q^{(r-s)j} \binom kju_{r - s}^j u_{m - r}^{k - j} w_{n - (r - s)k + (m - s)j} }  = q^{(r-s)k} u_{m - s}^k w_n\,,
\end{equation}
\begin{equation}\label{eq.kf3kgmr}
\sum_{j = 0}^k {( - 1)^j \binom kju_{r - s}^j u_{m - s}^{k - j} w_{n + (r - s)k + (m - r)j} }  = q^{(r-s)k} u_{m - r}^k w_n\,,
\end{equation}
\begin{equation}\label{eq.sa53jkd}
\sum_{j = 0}^k { (-1)^jq^{(r-s)(k-j)} \binom kju_{m + r}^j u_{m + s}^{k - j} w_{n - (m + r)k + (r - s)j} }  = (-1)^ku_{r - s}^k w_n\,,
\end{equation}
\begin{equation}
\sum_{j = 0}^k {q^{(r-s)j} \binom kju_{r - s}^j u_{m + s}^{k - j} w_{n - (r - s)k + (m + r)j} }  = q^{(r-s)k} u_{m + r}^k w_n
\end{equation}
and
\begin{equation}\label{eq.kh2azr9}
\sum_{j = 0}^k {( - 1)^j \binom kju_{r - s}^j u_{m + r}^{k - j} w_{n + (r - s)k + (m + s)j} }  = q^{(r-s)k} u_{m + s}^k w_n\,.
\end{equation}

\end{theorem}
\begin{proof}
To derive identities \eqref{eq.a4qiltd} -- \eqref{eq.kf3kgmr}, write identity (H) as
\[
u_{r - s} w_n  = u_{m - s} w_{n - (m - r)}  - q^{r - s} u_{m - r} w_{n - (m - s)}\,;
\]
identify $h=u_{r-s}$, $f_1=u_{m-s}$, $f_2=-q^{r-s}u_{m-r}$, $X_n=w_n$, $c=m-r$, $d=m-s$ and use these in Lemma \ref{lem.binomial}. Identities \eqref{eq.sa53jkd} -- \eqref{eq.kh2azr9} are obtained from identities \eqref{eq.a4qiltd} -- \eqref{eq.kf3kgmr} by interchanging $r$ and $-s$ and $s$ and $-r$.
\end{proof}
Particular cases of identities \eqref{eq.a4qiltd} -- \eqref{eq.kh2azr9} are the following summation identities involving only numbers from Lucas sequence of the first kind:
\begin{equation}
\sum_{j = 0}^k { (-1)^jq^{(r-s)(k-j)} \binom kju_{m - s}^j u_{m - r}^{k - j} u_{n - (m - s)k + (r - s)j} }  = (-1)^ku_{r - s}^k u_n\,,
\end{equation}
\begin{equation}
\sum_{j = 0}^k {q^{(r-s)j} \binom kju_{r - s}^j u_{m - r}^{k - j} u_{n - (r - s)k + (m - s)j} }  = q^{(r-s)k} u_{m - s}^k u_n\,,
\end{equation}
\begin{equation}
\sum_{j = 0}^k {( - 1)^j \binom kju_{r - s}^j u_{m - s}^{k - j} u_{n + (r - s)k + (m - r)j} }  = q^{(r-s)k} u_{m - r}^k u_n\,,
\end{equation}
\begin{equation}
\sum_{j = 0}^k { (-1)^jq^{(r-s)(k-j)} \binom kju_{m + r}^j u_{m + s}^{k - j} u_{n - (m + r)k + (r - s)j} }  = (-1)^ku_{r - s}^k u_n\,,
\end{equation}
\begin{equation}
\sum_{j = 0}^k {q^{(r-s)j} \binom kju_{r - s}^j u_{m + s}^{k - j} u_{n - (r - s)k + (m + r)j} }  = q^{(r-s)k} u_{m + r}^k u_n
\end{equation}
and
\begin{equation}\label{eq.gz6ate0}
\sum_{j = 0}^k {( - 1)^j \binom kju_{r - s}^j u_{m + r}^{k - j} u_{n + (r - s)k + (m + s)j} }  = q^{(r-s)k} u_{m + s}^k u_n\,;
\end{equation}
and the corresponding identities involving numbers from Lucas sequences of both kinds:
\begin{equation}
\sum_{j = 0}^k { (-1)^jq^{(r-s)(k-j)} \binom kju_{m - s}^j u_{m - r}^{k - j} v_{n - (m - s)k + (r - s)j} }  = (-1)^ku_{r - s}^k v_n\,,
\end{equation}
\begin{equation}
\sum_{j = 0}^k {q^{(r-s)j} \binom kju_{r - s}^j u_{m - r}^{k - j} v_{n - (r - s)k + (m - s)j} }  = q^{(r-s)k} u_{m - s}^k v_n\,,
\end{equation}
\begin{equation}
\sum_{j = 0}^k {( - 1)^j \binom kju_{r - s}^j u_{m - s}^{k - j} v_{n + (r - s)k + (m - r)j} }  = q^{(r-s)k} u_{m - r}^k v_n\,,
\end{equation}
\begin{equation}
\sum_{j = 0}^k { (-1)^jq^{(r-s)(k-j)} \binom kju_{m + r}^j u_{m + s}^{k - j} v_{n - (m + r)k + (r - s)j} }  = (-1)^ku_{r - s}^k v_n\,,
\end{equation}
\begin{equation}
\sum_{j = 0}^k {q^{(r-s)j} \binom kju_{r - s}^j u_{m + s}^{k - j} v_{n - (r - s)k + (m + r)j} }  = q^{(r-s)k} u_{m + r}^k v_n
\end{equation}
and
\begin{equation}
\sum_{j = 0}^k {( - 1)^j \binom kju_{r - s}^j u_{m + r}^{k - j} v_{n + (r - s)k + (m + s)j} }  = q^{(r-s)k} u_{m + s}^k v_n\,.
\end{equation}
\section{Summation identities not involving binomial coefficients}
\begin{theorem}\label{thm.c2gyoa1}
The following identities hold for $r$, $s$, $m$, $n$ and $k$ arbitrary integers:
\begin{equation}\label{eq.hlcov46}
\begin{split}
&u_{r - s} \sum_{j = 0}^k {q^{(s-r)j} w_{m + s}^{k - j} w_{m + r}^j w_{n - (r - s)k + m + s + (r - s)j} }\\
&\qquad= q^{(s-r)k} u_n w_{m + r}^{k + 1}  - q^{r - s} u_{n - (r - s)(k + 1)} w_{m + s}^{k + 1}
\end{split}
\end{equation}
and
\begin{equation}\label{eq.ao5nh45}
\begin{split}
&u_{r - s} \sum_{j = 0}^k {q^{(s-r)j} w_{m - r}^{k - j} w_{m - s}^j w_{n - (r - s)k + m - r + (r - s)j} }\\
&\qquad= q^{(s-r)k} u_n w_{m - s}^{k + 1}  - q^{r - s} u_{n - (r - s)(k + 1)} w_{m - r}^{k + 1}\,.
\end{split}
\end{equation}

\end{theorem}
\begin{proof}
To prove identity \eqref{eq.hlcov46}, write identity~(H) as
\begin{equation}
w_{m + r} u_n  = q^{r-s} w_{m + s} u_{n - (r - s)}  + u_{r - s} w_{n + m + s}\,,
\end{equation}
identify $h=w_{m+r}$, $f_1=q^{r-s}w_{m+s}$, $f_2=u_{r-s}$, $X_n=u_n$, $Y_n=w_{n+m+s}$, $c=r-s$ and $d=0$ and use these in Lemma \ref{lem.u4bqbkc}. Identity \eqref{eq.ao5nh45} is obtained from identity \eqref{eq.hlcov46} through the transformation $r\to -s$, $s\to -r$.
\end{proof}
\begin{theorem}\label{thm.ik24j18}
The following identities hold for arbitrary integers $r$, $s$, $n$, $m$ and $k$:
\begin{equation}
\begin{split}
&- q^{r - s} u_{m - r} \sum_{j = 0}^k {u_{m - s} ^{k - j} u_{r - s} ^j w_{n - (m - r)k - (m - s) + (m - r)j} }\\
&\qquad\qquad\qquad\qquad= u_{r - s} ^{k + 1} w_n  - u_{m - s} ^{k + 1} w_{n - (m - r)(k + 1)}\,,
\end{split}
\end{equation}
\begin{equation}
\begin{split}
&(-1)^ku_{m - s} \sum_{j = 0}^k { (-1)^jq^{(r-s)(k-j)} u_{m - r} ^{k - j} u_{r - s} ^j w_{n - (m - s)k - (m - r) + (m - s)j} }\\
&\qquad\qquad\qquad= u_{r - s} ^{k + 1} w_n  - (-1)^{k+1}q^{(r-s)(k+1)} u_{m - r} ^{k + 1} w_{n - (m - s)(k + 1)}\,,
\end{split}
\end{equation}
\begin{equation}
\begin{split}
&u_{r - s} \sum_{j = 0}^k {q^{(s-r)j} u_{m - r} ^{k - j} u_{m - s} ^j w_{n - (r - s)k + (m - r) + (r - s)j} }\\
&\qquad\qquad= q^{(s-r)k} u_{m - s} ^{k + 1} w_n  - q^{r - s} u_{m - r} ^{k + 1} w_{n - (r - s)(k + 1)}\,,
\end{split}
\end{equation}
\begin{equation}
\begin{split}
&- q^{r - s} u_{m + s} \sum_{j = 0}^k {u_{m + r} ^{k - j} u_{r - s} ^j w_{n - (m + s)k - (m + r) + (m + s)j} }\qquad\\
&\qquad\qquad\qquad\qquad= u_{r - s} ^{k + 1} w_n  - u_{m + r} ^{k + 1} w_{n - (m + s)(k + 1)}\,,
\end{split}
\end{equation}
\begin{equation}
\begin{split}
&(-1)^ku_{m + r} \sum_{j = 0}^k { (-1)^jq^{(r-s)(k-j)} u_{m + s} ^{k - j} u_{r - s} ^j w_{n - (m + r)k - (m + s) + (m + r)j} }\\
&\qquad\qquad\qquad= u_{r - s} ^{k + 1} w_n  - (-1)^{k+1}q^{(r-s)(k+1)} u_{m + s} ^{k + 1} w_{n - (m + r)(k + 1)}
\end{split}
\end{equation}
and
\begin{equation}
\begin{split}
&u_{r - s} \sum_{j = 0}^k {q^{(s-r)j} u_{m + s} ^{k - j} u_{m + r} ^j w_{n - (r - s)k + (m + s) + (r - s)j} }\\
&\qquad\qquad= q^{(s-r)k} u_{m + r} ^{k + 1} w_n  - q^{r - s} u_{m + s} ^{k + 1} w_{n - (r - s)(k + 1)}\,.
\end{split}
\end{equation}

\end{theorem}
\begin{proof}
In Lemma \ref{lem.s9jfs7n}, with $X_n=w_n$, use the $h$, $f_1$, $f_2$, $c$ and $d$ obtained in the proof of Theorem \ref{thm.tutbc4x}.
\end{proof}
In particular, we have the following summation identities involving only numbers from Lucas sequence of the first kind:
\begin{equation}
\begin{split}
&- q^{r - s} u_{m - r} \sum_{j = 0}^k {u_{m - s} ^{k - j} u_{r - s} ^j u_{n - (m - r)k - (m - s) + (m - r)j} }\\
&\qquad\qquad\qquad\qquad= u_{r - s} ^{k + 1} u_n  - u_{m - s} ^{k + 1} u_{n - (m - r)(k + 1)}\,,
\end{split}
\end{equation}
\begin{equation}
\begin{split}
&(-1)^ku_{m - s} \sum_{j = 0}^k { (-1)^jq^{(r-s)(k-j)} u_{m - r} ^{k - j} u_{r - s} ^j u_{n - (m - s)k - (m - r) + (m - s)j} }\\
&\qquad\qquad\qquad= u_{r - s} ^{k + 1} u_n  - (-1)^{k+1}q^{(r-s)(k+1)} u_{m - r} ^{k + 1} u_{n - (m - s)(k + 1)}\,,
\end{split}
\end{equation}
\begin{equation}
\begin{split}
&u_{r - s} \sum_{j = 0}^k {q^{(s-r)j} u_{m - r} ^{k - j} u_{m - s} ^j u_{n - (r - s)k + (m - r) + (r - s)j} }\\
&\qquad\qquad= q^{(s-r)k} u_{m - s} ^{k + 1} u_n  - q^{r - s} u_{m - r} ^{k + 1} u_{n - (r - s)(k + 1)}\,,
\end{split}
\end{equation}
\begin{equation}
\begin{split}
&- q^{r - s} u_{m + s} \sum_{j = 0}^k {u_{m + r} ^{k - j} u_{r - s} ^j u_{n - (m + s)k - (m + r) + (m + s)j} }\qquad\\
&\qquad\qquad\qquad\qquad= u_{r - s} ^{k + 1} u_n  - u_{m + r} ^{k + 1} u_{n - (m + s)(k + 1)}\,,
\end{split}
\end{equation}
\begin{equation}
\begin{split}
&(-1)^ku_{m + r} \sum_{j = 0}^k { (-1)^jq^{(r-s)(k-j)} u_{m + s} ^{k - j} u_{r - s} ^j u_{n - (m + r)k - (m + s) + (m + r)j} }\\
&\qquad\qquad\qquad= u_{r - s} ^{k + 1} u_n  - (-1)^{k+1}q^{(r-s)(k+1)} u_{m + s} ^{k + 1} u_{n - (m + r)(k + 1)}\,,
\end{split}
\end{equation}
and
\begin{equation}
\begin{split}
&u_{r - s} \sum_{j = 0}^k {q^{(s-r)j} u_{m + s} ^{k - j} u_{m + r} ^j u_{n - (r - s)k + (m + s) + (r - s)j} }\\
&\qquad\qquad= q^{(s-r)k} u_{m + r} ^{k + 1} u_n  - q^{r - s} u_{m + s} ^{k + 1} u_{n - (r - s)(k + 1)}\,;
\end{split}
\end{equation}
and the corresponding results involving numbers from Lucas sequences of both kinds:
\begin{equation}
\begin{split}
&- q^{r - s} u_{m - r} \sum_{j = 0}^k {u_{m - s} ^{k - j} u_{r - s} ^j v_{n - (m - r)k - (m - s) + (m - r)j} }\\
&\qquad\qquad\qquad\qquad= u_{r - s} ^{k + 1} v_n  - u_{m - s} ^{k + 1} v_{n - (m - r)(k + 1)}\,,
\end{split}
\end{equation}
\begin{equation}
\begin{split}
&(-1)^ku_{m - s} \sum_{j = 0}^k { (-1)^jq^{(r-s)(k-j)} u_{m - r} ^{k - j} u_{r - s} ^j v_{n - (m - s)k - (m - r) + (m - s)j} }\\
&\qquad\qquad\qquad= u_{r - s} ^{k + 1} v_n  - (-1)^{k+1}q^{(r-s)(k+1)} u_{m - r} ^{k + 1} v_{n - (m - s)(k + 1)}\,,
\end{split}
\end{equation}
\begin{equation}
\begin{split}
&u_{r - s} \sum_{j = 0}^k {q^{(s-r)j} u_{m - r} ^{k - j} u_{m - s} ^j v_{n - (r - s)k + (m - r) + (r - s)j} }\\
&\qquad\qquad= q^{(s-r)k} u_{m - s} ^{k + 1} v_n  - q^{r - s} u_{m - r} ^{k + 1} v_{n - (r - s)(k + 1)}\,.
\end{split}
\end{equation}
\begin{equation}
\begin{split}
&- q^{r - s} u_{m + s} \sum_{j = 0}^k {u_{m + r} ^{k - j} u_{r - s} ^j v_{n - (m + s)k - (m + r) + (m + s)j} }\qquad\\
&\qquad\qquad\qquad\qquad= u_{r - s} ^{k + 1} v_n  - u_{m + r} ^{k + 1} v_{n - (m + s)(k + 1)}\,,
\end{split}
\end{equation}
\begin{equation}
\begin{split}
&(-1)^ku_{m + r} \sum_{j = 0}^k { (-1)^jq^{(r-s)(k-j)} u_{m + s} ^{k - j} u_{r - s} ^j v_{n - (m + r)k - (m + s) + (m + r)j} }\\
&\qquad\qquad\qquad= u_{r - s} ^{k + 1} v_n  - (-1)^{k+1}q^{(r-s)(k+1)} u_{m + s} ^{k + 1} v_{n - (m + r)(k + 1)}\,,
\end{split}
\end{equation}
and
\begin{equation}
\begin{split}
&u_{r - s} \sum_{j = 0}^k {q^{(s-r)j} u_{m + s} ^{k - j} u_{m + r} ^j v_{n - (r - s)k + (m + s) + (r - s)j} }\\
&\qquad\qquad= q^{(s-r)k} u_{m + r} ^{k + 1} v_n  - q^{r - s} u_{m + s} ^{k + 1} v_{n - (r - s)(k + 1)}\,.
\end{split}
\end{equation}
\section{Summation identities involving reciprocals}
\begin{theorem}
The following identities hold for values of $r$, $s$, $m$, $n$, $k$ for which the summand is non-singular in the summation interval:
\begin{equation}\label{eq.vd3kfav}
\begin{split}
u_nu_{n - (r - s)(k + 1)}u_{r - s} &\sum_{j = 0}^k {q^{(r-s)j} \frac{{w_{m + r}^{k - j} w_{m + s}^j w_{n + m + s - (r - s)k + (r - s)j} }}{{u_{n - (r - s)k + (r - s)j} u_{n - r + s - (r - s)k + (r - s)j} }}}\\
&\qquad\qquad= u_nw_{m + r}^{k + 1} - q^{(r-s)(k+1)} u_{n - (r - s)(k + 1)}w_{m + s}^{k + 1}\,,
\end{split}
\end{equation}
\begin{equation}\label{eq.jwgeagg}
\begin{split}
u_nu_{n - (r - s)(k + 1)}u_{r - s} &\sum_{j = 0}^k {q^{(r-s)j} \frac{{w_{m - s}^{k - j} w_{m - r}^j w_{n + m - r - (r - s)k + (r - s)j} }}{{u_{n - (r - s)k + (r - s)j} u_{n + s - r - (r - s)k + (r - s)j} }}}\\
&\qquad\qquad= u_nw_{m - s}^{k + 1} - q^{(r-s)(k+1)} u_{n - (r - s)(k + 1)}w_{m - r}^{k + 1}\,.
\end{split}
\end{equation}

\end{theorem}
\begin{proof}
In Lemma \ref{lem.reciprocal1}, make the identification $X_n=u_n$ and $Y_n=w_{n+m+s}$ and use the $f_1$, $f_2$, $h$, $c$ and $d$ obtained in the proof of Theorem \ref{thm.c2gyoa1}.
\end{proof}
Particular cases of identities \eqref{eq.vd3kfav} and \eqref{eq.jwgeagg} are the following:
\begin{equation}
\begin{split}
u_nu_{n - (r - s)(k + 1)}u_{r - s} &\sum_{j = 0}^k {q^{(r-s)j} \frac{{u_{m + r}^{k - j} u_{m + s}^j u_{n + m + s - (r - s)k + (r - s)j} }}{{u_{n - (r - s)k + (r - s)j} u_{n - r + s - (r - s)k + (r - s)j} }}}\\
&\qquad\qquad= u_nu_{m + r}^{k + 1} - q^{(r-s)(k+1)} u_{n - (r - s)(k + 1)}u_{m + s}^{k + 1}\,,
\end{split}
\end{equation}
\begin{equation}
\begin{split}
u_nu_{n - (r - s)(k + 1)}u_{r - s} &\sum_{j = 0}^k {q^{(r-s)j} \frac{{u_{m - s}^{k - j} u_{m - r}^j u_{n + m - r - (r - s)k + (r - s)j} }}{{u_{n - (r - s)k + (r - s)j} u_{n + s - r - (r - s)k + (r - s)j} }}}\\
&\qquad\qquad= u_nu_{m - s}^{k + 1} - q^{(r-s)(k+1)} u_{n - (r - s)(k + 1)}u_{m - r}^{k + 1}\,;
\end{split}
\end{equation}
and
\begin{equation}
\begin{split}
u_nu_{n - (r - s)(k + 1)}u_{r - s} &\sum_{j = 0}^k {q^{(r-s)j} \frac{{v_{m + r}^{k - j} v_{m + s}^j v_{n + m + s - (r - s)k + (r - s)j} }}{{u_{n - (r - s)k + (r - s)j} u_{n - r + s - (r - s)k + (r - s)j} }}}\\
&\qquad\qquad= u_nv_{m + r}^{k + 1} - q^{(r-s)(k+1)} u_{n - (r - s)(k + 1)}v_{m + s}^{k + 1}\,,
\end{split}
\end{equation}
\begin{equation}
\begin{split}
u_nu_{n - (r - s)(k + 1)}u_{r - s} &\sum_{j = 0}^k {q^{(r-s)j} \frac{{v_{m - s}^{k - j} v_{m - r}^j v_{n + m - r - (r - s)k + (r - s)j} }}{{u_{n - (r - s)k + (r - s)j} u_{n + s - r - (r - s)k + (r - s)j} }}}\\
&\qquad\qquad= u_nv_{m - s}^{k + 1} - q^{(r-s)(k+1)} u_{n - (r - s)(k + 1)}v_{m - r}^{k + 1}\,.
\end{split}
\end{equation}
\begin{theorem}
The following identities hold for values of $r$, $s$, $m$, $n$, $k$ for which the summand is non-singular in the summation interval:
\begin{equation}
\begin{split}
- q^{r - s} u_{m - r} w_n w_{n - (m - r)(k + 1)} &\sum_{j = 0}^k {\frac{{u_{r - s}^{k - j} u_{m - s}^j w_{n - m + s - (m - r)k + (m - r)j} }}{{w_{n - (m - r)k + (m - r)j} w_{n - (m - r) - (m - r)k + (m - r)j} }}}\\
&\quad= u_{r - s} ^{k + 1} w_n  - u_{m - s}^{k + 1} w_{n - (m - r)(k + 1)}\,,
\end{split}
\end{equation}
\begin{equation}
\begin{split}
u_{m - s} w_n w_{n - (m - s)(k + 1)} &\sum_{j = 0}^k {\frac{(-1)^jq^{(r-s)j}u_{r - s}^{k - j} { u_{m - r}}^j w_{n - (m - r) - (m - s)k + (m - s)j} }{{w_{n - (m - s)k + (m - s)j} w_{n - (m - s) - (m - s)k + (m - s)j} }}}\\
&\quad= u_{r - s} ^{k + 1} w_n  - (-1)^{k+1}q^{(r-s)(k+1)}u_{m - r}^{k + 1} w_{n - (m - s)(k + 1)}\,,
\end{split}
\end{equation}
\begin{equation}
\begin{split}
u_{r-s}w_nw_{n - (r - s)(k + 1)}&\sum_{j = 0}^k {\frac{q^{(r-s)j}u_{m - s}^{k-j}u_{m-r}^jw_{n + m - r - (r - s)k + (r - s)j} }{{w_{n - (r - s)k + (r - s)j} w_{n - (r - s) - (r - s)k + (r - s)j} }}}\\
&\qquad= u_{m - s}^{k+1}w_n -q^{(r-s)(k+1)}u_{m-r}^{k+1}w_{n - (r - s)(k + 1)}\,,
\end{split}
\end{equation}
\begin{equation}
\begin{split}
- q^{r - s} u_{m + s} w_n w_{n - (m + s)(k + 1)} &\sum_{j = 0}^k {\frac{{u_{r - s}^{k - j} u_{m + r}^j w_{n - m - r - (m + s)k + (m + s)j} }}{{w_{n - (m + s)k + (m + s)j} w_{n - (m + s) - (m + s)k + (m + s)j} }}}\\
&\quad= u_{r - s} ^{k + 1} w_n  - u_{m + r}^{k + 1} w_{n - (m + s)(k + 1)}\,,
\end{split}
\end{equation}
\begin{equation}
\begin{split}
u_{m + r} w_n w_{n - (m + r)(k + 1)} &\sum_{j = 0}^k {\frac{(-1)^jq^{(r-s)j}u_{r - s}^{k - j} { u_{m + s}}^j w_{n - (m + s) - (m + r)k + (m + r)j} }{{w_{n - (m + r)k + (m + r)j} w_{n - (m + r) - (m + r)k + (m + r)j} }}}\\
&\quad= u_{r - s} ^{k + 1} w_n  - (-1)^{k+1}q^{(r-s)(k+1)}u_{m + s}^{k + 1} w_{n - (m + r)(k + 1)}\,,
\end{split}
\end{equation}
and
\begin{equation}
\begin{split}
u_{r - s}w_nw_{n - (r - s)(k + 1)}&\sum_{j = 0}^k {\frac{q^{(r-s)j}u_{m + r}^{k-j}u_{m + s}^jw_{n + m + s - (r - s)k + (r - s)j} }{{w_{n - (r - s)k + (r - s)j} w_{n - (r - s) - (r - s)k + (r - s)j} }}}\\
&\qquad= u_{m + r}^{k+1}w_n -q^{(r-s)(k+1)}u_{m + s}^{k+1}w_{n - (r - s)(k + 1)}\,.
\end{split}
\end{equation}

\end{theorem}
\begin{proof}
In Lemma \ref{lem.reciprocal}, make the identification $X_n=w_n$ and use the $f_1$, $f_2$, $h$, $c$ and $d$ obtained in the proof of Theorem \ref{thm.ik24j18}.
\end{proof}
In particular, we have the summation identities involving only numbers from Lucas sequence of the first kind:
\begin{equation}
\begin{split}
- q^{r - s} u_{m - r} u_n u_{n - (m - r)(k + 1)} &\sum_{j = 0}^k {\frac{{u_{r - s}^{k - j} u_{m - s}^j u_{n - m + s - (m - r)k + (m - r)j} }}{{u_{n - (m - r)k + (m - r)j} u_{n - (m - r) - (m - r)k + (m - r)j} }}}\\
&\quad= u_{r - s} ^{k + 1} u_n  - u_{m - s}^{k + 1} u_{n - (m - r)(k + 1)}\,,
\end{split}
\end{equation}
\begin{equation}
\begin{split}
u_{m - s} u_n u_{n - (m - s)(k + 1)} &\sum_{j = 0}^k {\frac{(-1)^jq^{(r-s)j}u_{r - s}^{k - j} { u_{m - r}}^j u_{n - (m - r) - (m - s)k + (m - s)j} }{{u_{n - (m - s)k + (m - s)j} u_{n - (m - s) - (m - s)k + (m - s)j} }}}\\
&\quad= u_{r - s} ^{k + 1} u_n  - (-1)^{k+1}q^{(r-s)(k+1)}u_{m - r}^{k + 1} u_{n - (m - s)(k + 1)}\,,
\end{split}
\end{equation}
\begin{equation}
\begin{split}
u_{r-s}u_nu_{n - (r - s)(k + 1)}&\sum_{j = 0}^k {\frac{q^{(r-s)j}u_{m - s}^{k-j}u_{m-r}^ju_{n + m - r - (r - s)k + (r - s)j} }{{u_{n - (r - s)k + (r - s)j} u_{n - (r - s) - (r - s)k + (r - s)j} }}}\\
&\qquad= u_{m - s}^{k+1}u_n -q^{(r-s)(k+1)}u_{m-r}^{k+1}u_{n - (r - s)(k + 1)}\,,
\end{split}
\end{equation}
\begin{equation}
\begin{split}
- q^{r - s} u_{m + s} u_n u_{n - (m + s)(k + 1)} &\sum_{j = 0}^k {\frac{{u_{r - s}^{k - j} u_{m + r}^j u_{n - m - r - (m + s)k + (m + s)j} }}{{u_{n - (m + s)k + (m + s)j} u_{n - (m + s) - (m + s)k + (m + s)j} }}}\\
&\quad= u_{r - s} ^{k + 1} u_n  - u_{m + r}^{k + 1} u_{n - (m + s)(k + 1)}\,,
\end{split}
\end{equation}
\begin{equation}
\begin{split}
u_{m + r} u_n u_{n - (m + r)(k + 1)} &\sum_{j = 0}^k {\frac{(-1)^jq^{(r-s)j}u_{r - s}^{k - j} { u_{m + s}}^j u_{n - (m + s) - (m + r)k + (m + r)j} }{{u_{n - (m + r)k + (m + r)j} u_{n - (m + r) - (m + r)k + (m + r)j} }}}\\
&\quad= u_{r - s} ^{k + 1} u_n  - (-1)^{k+1}q^{(r-s)(k+1)}u_{m + s}^{k + 1} u_{n - (m + r)(k + 1)}\,,
\end{split}
\end{equation}
and
\begin{equation}
\begin{split}
u_{r - s}u_nu_{n - (r - s)(k + 1)}&\sum_{j = 0}^k {\frac{q^{(r-s)j}u_{m + r}^{k-j}u_{m + s}^ju_{n + m + s - (r - s)k + (r - s)j} }{{u_{n - (r - s)k + (r - s)j} u_{n - (r - s) - (r - s)k + (r - s)j} }}}\\
&\qquad= u_{m + r}^{k+1}u_n -q^{(r-s)(k+1)}u_{m + s}^{k+1}u_{n - (r - s)(k + 1)}\,;
\end{split}
\end{equation}
and the corresponding identities involving numbers from Lucas sequences of both kinds:
\begin{equation}
\begin{split}
- q^{r - s} u_{m - r} v_n v_{n - (m - r)(k + 1)} &\sum_{j = 0}^k {\frac{{u_{r - s}^{k - j} u_{m - s}^j v_{n - m + s - (m - r)k + (m - r)j} }}{{v_{n - (m - r)k + (m - r)j} v_{n - (m - r) - (m - r)k + (m - r)j} }}}\\
&\quad= u_{r - s} ^{k + 1} v_n  - u_{m - s}^{k + 1} v_{n - (m - r)(k + 1)}\,,
\end{split}
\end{equation}
\begin{equation}
\begin{split}
u_{m - s} v_n v_{n - (m - s)(k + 1)} &\sum_{j = 0}^k {\frac{(-1)^jq^{(r-s)j}u_{r - s}^{k - j} { u_{m - r}}^j v_{n - (m - r) - (m - s)k + (m - s)j} }{{v_{n - (m - s)k + (m - s)j} v_{n - (m - s) - (m - s)k + (m - s)j} }}}\\
&\quad= u_{r - s} ^{k + 1} v_n  - (-1)^{k+1}q^{(r-s)(k+1)}u_{m - r}^{k + 1} v_{n - (m - s)(k + 1)}\,,
\end{split}
\end{equation}
\begin{equation}
\begin{split}
u_{r-s}v_nv_{n - (r - s)(k + 1)}&\sum_{j = 0}^k {\frac{q^{(r-s)j}u_{m - s}^{k-j}u_{m-r}^jv_{n + m - r - (r - s)k + (r - s)j} }{{v_{n - (r - s)k + (r - s)j} v_{n - (r - s) - (r - s)k + (r - s)j} }}}\\
&\qquad= u_{m - s}^{k+1}v_n -q^{(r-s)(k+1)}u_{m-r}^{k+1}v_{n - (r - s)(k + 1)}\,,
\end{split}
\end{equation}
\begin{equation}
\begin{split}
- q^{r - s} u_{m + s} v_n v_{n - (m + s)(k + 1)} &\sum_{j = 0}^k {\frac{{u_{r - s}^{k - j} u_{m + r}^j v_{n - m - r - (m + s)k + (m + s)j} }}{{v_{n - (m + s)k + (m + s)j} v_{n - (m + s) - (m + s)k + (m + s)j} }}}\\
&\quad= u_{r - s} ^{k + 1} v_n  - u_{m + r}^{k + 1} v_{n - (m + s)(k + 1)}\,,
\end{split}
\end{equation}
\begin{equation}
\begin{split}
u_{m + r} v_n v_{n - (m + r)(k + 1)} &\sum_{j = 0}^k {\frac{(-1)^jq^{(r-s)j}u_{r - s}^{k - j} { u_{m + s}}^j v_{n - (m + s) - (m + r)k + (m + r)j} }{{v_{n - (m + r)k + (m + r)j} v_{n - (m + r) - (m + r)k + (m + r)j} }}}\\
&\quad= u_{r - s} ^{k + 1} v_n  - (-1)^{k+1}q^{(r-s)(k+1)}u_{m + s}^{k + 1} v_{n - (m + r)(k + 1)}
\end{split}
\end{equation}
and
\begin{equation}
\begin{split}
u_{r - s}v_nv_{n - (r - s)(k + 1)}&\sum_{j = 0}^k {\frac{q^{(r-s)j}u_{m + r}^{k-j}u_{m + s}^jv_{n + m + s - (r - s)k + (r - s)j} }{{v_{n - (r - s)k + (r - s)j} v_{n - (r - s) - (r - s)k + (r - s)j} }}}\\
&\qquad= u_{m + r}^{k+1}v_n -q^{(r-s)(k+1)}u_{m + s}^{k+1}v_{n - (r - s)(k + 1)}\,.
\end{split}
\end{equation}

\hrule

\noindent 2010 {\it Mathematics Subject Classification}:
Primary 11B39; Secondary 11B37.

\noindent \emph{Keywords: }
Horadam sequence, Fibonacci number, Lucas number, Lucas sequence, summation identity.

\hrule




\begin{thebibliography}{99}
\bibitem{adegoke18} Kunle Adegoke, Weighted sums of some second-order sequences, \emph{The Fibonacci Quarterly} {\bf 56}:3 (2018), 252--262.

\bibitem{horadam65} A.~F.~Horadam, Basic properties of a certain generalized sequence of numbers, \emph{The Fibonacci Quarterly} {\bf 3}:3 (1965), 161--176.


\bibitem{ribenboim} P. Ribenboim, \emph{My numbers, my friends}, Springer-Verlag, New York, (2000).

\bibitem{mathworld_lucas} E. W. Weisstein, \htmladdnormallink{Lucas Sequence}{http://mathworld.wolfram.com/LucasSequence.html}, \emph{MathWorld--A Wolfram Web Resource}, March 2019.

\bibitem{wiki_lucas} Wikipedia contributors, \htmladdnormallink{Lucas sequence}{https://en.wikipedia.org/wiki/Lucas_sequence}, \emph{Wikipedia, The Free Encyclopedia}, November 2018.

\bibitem{koshy} T.~Koshy, \emph{Fibonacci and Lucas numbers with applications}, Wiley-Interscience, (2001).

\bibitem{vajda} S.~Vajda, \emph{Fibonacci and Lucas numbers, and the golden section: theory and applications}, Dover Press, (2008).


\end{thebibliography}
\end{document}